\documentclass[10pt]{amsart}
\usepackage{times,amsmath,amsbsy,amssymb,amscd,mathrsfs}
\usepackage{diagbox}

\usepackage{graphicx,subfigure,epstopdf,wrapfig,chemarrow}

\usepackage{algorithm2e} 
\usepackage{algorithmic}

\usepackage{multicol,multirow}
\usepackage{mathtools}
\usepackage[usenames,dvipsnames,svgnames,table]{xcolor}
\usepackage[all]{xy}
\usepackage{wrapfig}

\definecolor{myBlue}{rgb}{0.0,0.0,0.55}
\usepackage[pdftex,colorlinks=true,citecolor=myBlue,linkcolor=myBlue]{hyperref}

\usepackage[hyperpageref]{backref}

\usepackage{comment,enumerate,multicol,xspace}

\newtheorem{theorem}{Theorem}[section]

\newtheorem{corollary}[theorem]{Corollary}

\newtheorem{definition}[theorem]{Definition}

\newtheorem{remark}[theorem]{Remark}

\newcommand{\mbb}{\mathbb}

\DeclareMathOperator*{\var}{Var}
\DeclareMathOperator{\rank}{rank}

\newcommand{\vertiii}[1]{{\left\vert\kern-0.25ex\left\vert\kern-0.25ex\left\vert #1 
    \right\vert\kern-0.25ex\right\vert\kern-0.25ex\right\vert}}

\begin{document}
\title{A Preconditioner based on Non-uniform Row Sampling for Linear Least Squares Problems}
\author{Long Chen and Huiwen Wu}
\date{\today}
\begin{abstract}
Least squares method is one of the simplest and most popular techniques applied in data fitting, imaging processing and high dimension data analysis. The classic methods like QR and SVD decomposition for solving least squares problems has a large computational cost. Iterative methods such as CG and Kaczmarz can reduce the complexity if the matrix is well conditioned but failed for the ill conditioned cases. Preconditioner based on randomized row sampling algorithms have been developed but destroy the sparsity. 
In this paper, a new preconditioner is constructed by non-uniform row sampling with a probability proportional to the squared norm of rows. Then Gauss Seidel iterations are applied to the normal equation of the sampled matrix which aims to grab the high frequency component of solution. After all, PCG is used to solve the normal equation with this preconditioner. Our preconditioner can keep the sparsity and improve the poor conditioning for highly overdetermined matrix. 
Experimental studies are presented on several different simulations including dense Gaussian matrix, `semi Gaussian' matrix, sparse random matrix, `UDV' matrix, and random graph Laplacian matrix to show the effectiveness of the proposed least square solver.
\end{abstract}
\maketitle

%------------------------------------------------------------------------------------------------------
\section{Introduction}\label{sec1}

Least squares method is one of the simplest and most commonly applied techniques of data fitting. It can be applied in statistics to construct linear regression model and unbiased linear estimator~\cite{neter1996applied}, in imaging processing for image deblurring~\cite{IMID1990}, and in high-dimensional data analysis like canonical polyadic tensor decomposition~\cite{TDA2009} etc.

Consider the overdetermined system
\begin{equation}\label{AXB}
A x =b ,
\end{equation}
where $A \in \mathbb R^{m \times n}, b \in \mathbb R^{m}, m \gg n, $  and $\rank(A) = n$. As $m > n$, solutions to~\eqref{AXB} are in general not unique. 
The least squares solution to the overdetermined system~\eqref{AXB} is
\begin{equation}\label{OPT}
x_{\rm{opt}} = \arg \min_{x} \| Ax -b\|_2^2 ,
\end{equation}
where $\| \cdot \|_2$ is the $l^2$ norm of a vector. 
It is not difficult to show that the least squares solution $x_{\rm{opt}}$ satisfies the so-called normal equation
\begin{equation}\label{NOR}
A^{\intercal} A \, x = A^{\intercal} b.  
\end{equation}

We shall develop a fast iterative least squares solver for~\eqref{AXB}.  There are various methods for solving the least squares problems. Classical methods such as QR and SVD decomposition that require $O(mn^2)$ operations which is prohibitive for large size problems~\cite{NLA2012}. Kaczmarz methods~\cite{KAC1937} and randomized Kaczmarz methods~\cite{RKEC2009} are effective iterative methods for consistent least squares problems while the computational costs highly depends on a scaled condition number of matrix $A$. Several fast least squares solvers have been developed recently utilizing randomization see, for example~\cite{FLSA2007, SRFT2008, BLEN2010}. 
These fast least squares solvers try to construct a spectrally equivalent but of smaller size matrix via random transformation and random sampling or mixing. Due to the random transformation, however, these methods destroy the sparsity of $A$, and thus may not be suitable for sparse matrices. 

Conjugate Gradient (CG) method is an efficient iterative algorithm for solving linear systems of equations when the matrix is symmetric and positive-definite~\cite{NLA2012}. Thus CG method can be applied to solve the normal equation~\eqref{NOR}. CG is the archetype of Krylov subspace method that only needs the matrix-vector multiplication $Au$ and $A^{\intercal}v$. In order to achieve accuracy $\epsilon$, it needs $k \approx O( |\log(\epsilon)|  \kappa(A))$ iterations, where the condition number $
\kappa(A) = \sigma_{\max} (A)/\sigma_{\min} (A)$
with $\sigma_{\max} (A), \sigma_{\min} (A)$ being the maximal and minimal (non-zero) singular values of $A$ respectively. 
When $\kappa(A)$ is large, the method converges very slowly. Preconditioned conjugate gradient (PCG) with proper preconditioners can be applied to accelerate the convergence. The key of success of PCG is a good preconditioner.  

We shall use row sampling to construct an efficient preconditioner. Firstly we apply the random sampling to the rows of matrix $A$ and select $C n \log(n)$ rows, with probability proportional to the squared norm of each row, to get a sampled matrix $A_s$. Then the preconditioner $P$ is constructed by solving the sampled normal equation 
\begin{equation}\label{SNO}
 A_s^{\intercal} A_s e =r. 
\end{equation}
via a few symmetric Gauss-Seidel iterations. At last, we equip PCG method with this preconditioner to solve the original normal equation~\eqref{NOR}.

The main motivation of our method is to utilize the benefits of random sampling while keep the sparsity. By the random row sampling, we got a much smaller matrix $A_s$ compared to $A$ which reduces the complexity of matrix multiplication. The sampling algorithm and the chosen sampling density ensure the sampled matrix will capture the high frequency part of the original matrix $A$. The Gauss-Seidel iterations to the sampled normal system~\eqref{SNO} will smooth out the high frequency of the error and thus improves the condition number.

 The complexity of our method is $O(n^3 \log(n)) + O(kmn) $ for dense matrices. For sparse matrices, suppose the number of nonzero elements of matrix $A$ is $nnz$ and $nnz \ll m n$. The complexity is at most $O(n^2) + O(k nnz)$ and could be $O(n) + O(k nnz)$ depending on the sparse pattern, where $k = O(|\log (\epsilon)| \kappa(PA))$ is the iteration steps of PCG depending on the condition number of the preconditioned system $PA$. Although we cannot obtain an uniform bound of $\kappa(PA)$, we show by the numerical experiments that PCG with the new preconditioner improves the convergence significantly comparing with CG with a simple diagonal preconditioner. 

The paper is organized as follows. Section 2 contains a summary of existing methods for least squares problem including classical methods like QR and SVD decompositions, Kaczmarz methods of cyclic and randomized version, and several fast least squares solvers. The details of our algorithms including how to construct the row sampling preconditioner and the complexity of algorithms are discussed in Section 3. Then numerical experiments are presented in Section 4, showing our proposed solver is efficient and effective. 

%-----------------------------------------------------------------------------------------------------------------
\section{Existing Methods}

In this section, we review several existing methods for least squares problem. These methods are: the most basic ones like QR and SVD decomposition, the iterative methods such as Conjugate Gradient Descent(CG), Kaczmarz methods and also the newly developed randomized methods like randomized Kaczmarz method, and fast least squares solvers with random transformation. Convergence rate and computational complexity are compared while complexity is mainly calculated for dense matrices. Each method has its advantages and disadvantages -- none is perfect. Due to the eruption of informations, the tendency is to find more efficient methods with less computation complexity and storage.

\subsection{Classical Methods}
Classical methods to solve the least squares problem include direct methods such as QR decomposition, SVD decomposition, and iterative methods such as CG methods and Kaczmarz methods. The complexity for QR decomposition and SVD decomposition is $O(mn^2)$ which is huge when $m$ or $n$ is large. This is the limitation of direct methods like QR and SVD decomposition. 

CG method is an iterative method to solve symmetric and positive definite (SPD) systems. It is applicable to sparse systems that are too large to be solved by a direct solver. In general for Krylov subspace methods, only matrix-vector product instead of matrix-matrix product is required and thus save the storage and reduce the complexity provided the method convergences fast. In the overdetermined case, if CG is applied to the normal equation~\eqref{NOR}, only the matrix-vector multiplication $Au$ and $A^{\intercal}v$ is needed which cost $O(nnz(A))$ operations. Here $nnz(\cdot)$ is the number of nonzero entries of a matrix. In order to achieve the accuracy $\epsilon$, CG needs $O(|\log(\epsilon)|\sqrt{\kappa(A^{\intercal}A)}) = O(|\log(\epsilon)|\kappa(A))$ steps, where the condition number of matrix $A$ is defined by
\begin{equation}\label{COND}
\kappa(A) = \frac{\sigma_{\max} (A)}{\sigma_{\min} (A)}
\end{equation} 
with $\sigma_{\max} (A), \sigma_{\min} (A)$ being the maximal and minimal singular values of $A$ respectively. Therefore the complexity for CG applied to ~\eqref{NOR} is $O(m n \kappa(A) |\log(\epsilon)|)$ for dense matrices and $O(nnz(A) \kappa(A) |\log(\epsilon)|)$ for sparse matrices. Tailored implementations of CG to normal equation~\eqref{NOR} include CGLS~\cite{bjorck1996numerical} and LSQR~\cite{LSQR1982}.

For consistent systems, Kaczmarz method can be applied. 
A linear system is called consistent if there is at least one solution, i.e. $b \in \rm{range}(A)$ in~\eqref{AXB}. Kaczmarz method is to project approximation onto the hyperplane $a_i x = b_{i}$ where $a_i$ is the $i$-th row of matrix $A$. 
Oswald and Zhou~\cite{CON2015} obtained the convergence rate of cyclic Kaczmarz method 
 $$
 \| x_{\rm{opt}} - x_s \|_2^2 \leq \left[1 - \frac{1}{(\log(n)+1)\kappa(A^{\intercal} D^{-1} A)}\right]^s \| x_{\rm{opt}} - x_0\|_2^2,
 $$
 where $x_{\rm{opt}}$ is the least squares solution, $x_s$ is the $s$th iterate consisting of sweeping of all rows, $x_0$ is the initial guess and $D$ is an $m \times m$ diagonal matrix which induces a row scaling. 
One sweeping takes $O( mn)$ operations. In order to achieve accuracy $\epsilon$, the total complexity of Kaczmarz method is $  O (m n \log(n) \kappa(A^{\intercal}D^{-1}A) | \log(\epsilon)| ). $

The advantage of iterative methods is that they utilize the sparsity since in each iteration only matrix-vector multiplications are calculated. For example,  the complexity of CG is $O(nnz(A) \kappa(A) |\log(\epsilon)|)$ when $A$ is sparse. For Kaczmarz method, the complexity also reduces since the cost of each projection is less than $O(n)$ depending on the sparsity of $A$. As the iteration steps of both CG and Kaczmarz for consistent systems depend crucially on the condition number $\kappa(A)$, they are slow if $A$ is ill-conditioned, i.e. $\kappa(A)\gg 1$. Preconditioner can be used to improve the condition number and in turn accelerate the convergence. One way to construct effective preconditioners is to use random sampling and random transformation, which will be discussed below.

\subsection{Randomized Methods}
There are several approaches to accelerate traditional least squares solvers via randomization. The main idea is either use random sampling or random projection to reduce the size of the original matrix, e.g. the randomized Kaczmarz method~\cite{RKEC2009} and randomized fast solvers in~\cite{BLEN2010,FLSA2011}, or construct preconditioners by random sampling to reduce the condition number which enable to apply PCG~\cite{SRFT2008} or LSQR~\cite{BLEN2010} of the preconditioned system.

\subsubsection{Randomized Kaczmarz Methods}
The convergence rate of the cyclic Kaczmarz method highly depends on the ordering of rows of $A$. In order to achieve a faster convergence which is independent of ordering of rows, choosing rows at random is a good strategy. For consistent systems, i.e. $b \in \text{Range}(A)$, Strohmer and Vershynin~\cite{RKEC2009}  proposed a randomized Kaczmarz (RK) method which selects the hyperplane to do projection via the probability proportional to $\| a_i \|_2^2$, and proved its exponential convergence in expectation, i.e.
$$\mathbb E (\| x_k - x_{\rm{opt}} \|_2^2) \leq (1 - \kappa_F(A)^{-2})^k \| x_0 - x_{\rm{opt}}\|_2^2,$$
where $x_k$ is the $k$-th iteration consisting of one projection only, $x_0$ is the initial guess, $x_{\rm{opt}}$ is the least squares solution and $\kappa_F(A) = \| A\|_F^2 \| (A^{\intercal} A)^{-1} \|_2^2  $ is a scaled condition number. Notice that here one iteration requires only $O(n)$ operations. To achieve the accuracy $\epsilon$, the expected iteration steps $O( \kappa^2_F(A) | \log(\epsilon)| )$ and the total expected complexity is $O(n  \kappa^2_F(A) | \log(\epsilon)|)$. 

For consistent systems, the randomized Kaczmarz method converges with expected exponential rate independent on the number of equations in the system. Indeed, the solver does not need to know the whole system but only a $O(n\log n)$ rows as the system is assumed to be consistent~\cite{RKEC2009}. Thus it outperforms some traditional methods like CG on general extremely overdetermined system. The main limitation of RK is its inability of handling inconsistent systems. For instance, to solve $Ax = b$, where $b = y+w,$ with $y = b_{R(A)}$ is the projection of $b$ onto range of $A$ and $w = b_{R(A)^{\perp}}$, the randomized Kaczmarz method is effective when least squares estimation is effective, i.e. the least squares error $\| w \|_2$ is negligible~\cite{REK2013}. Extension of randomized Kaczmarz methods to inconsistent systems can be found in~\cite{RKN2010,REK2013,RKI2015}. 

\subsubsection{Fast Least Squares Solvers by Random Transformations}
Drineas, Mahoney, Muthukrishnan and  Sarlos~\cite{FLSA2007} developed two randomized algorithms for least square problems. Instead of solving the original least squares problem $\| A x - b \|_2^2$, they solve an approximate least squares problem $\| X A x - X b \|_2^2$, where $X = SHD$ for randomized sampling  or $X = THD$ for randomized projection. The operator $HD$ is the randomized Hadamard transformation which aims to spread out the elements of $A$ and $S$ is the uniform sampling matrix and $T$ is the randomized projection matrix aiming to reduce the size of the original problem. The complexity of Hadamard transformation is $O(m \log(m))$ and the complexity of traditional methods to the approximated least squares problem is $O(r n^2)$ with the sample size $r = O(n /\epsilon)$ is chosen so that $XA$ is full rank but $r \ll m$. The complexity reduce to $O(mn\log(n/\epsilon) + n^3/\epsilon) \ll O(mn^2)$ provided $\epsilon$ is not too small ~\cite{FLSA2011}. 

Rokhlin and Tygert~\cite{SRFT2008} proposed a fast randomized algorithm for overdetermined linear least squares regression. They constructed subsampled randomized Fourier transform (SRFT) matrix $T$ of size $r \times m$ and then apply pivoted QR decomposition to $TA = QR\Pi$, with a $r \times n$ orthonormal matrix $Q$, an upper-triangular $n \times n$ matrix $R$ and an $n \times n$ permutation matrix $\Pi$. Then apply PCG with the right preconditioning matrix $P = R \Pi$, i.e. to minimize the preconditioned system $\| A P^{-1} y - b\|$.  
           
According to  theory developed in~\cite{SRFT2008}, $\kappa(AP^{-1}) = \kappa(TU)$ where the columns of $U$ are left singular vectors of $A$. The condition number of $TU$ can be controlled by the number of rows of $T$, i.e. $r$. In practice, $\kappa(TU) \leq 3$ when $r = 4n$. In this method, it converges fast since $\kappa(A P^{-1})$ is much smaller than $\kappa(A)$ but needs one QR decomposition which is $O(n^2 r)$. The total theoretical complexity is $O((\log(r) + \kappa(AP^{-1})|\log( \epsilon )| mn ) + O(n^2r)$. 

Avron, Maymounkov and Toledo~\cite{BLEN2010} developed an algorithm called BLENDENPIK, which supercharges LAPACK's dense least-squares solver. They introduce the concept of coherence number $\mu(A)$, which is the maximum of squared norm of rows of $Q$, where the columns of $Q$ are a set of orthonormal bases of range of $A$. They find that the uniform sampling will work if the coherence number of the matrix is small and apply  the row mixing to reduce $\mu(A)$ if it is large~\cite{BLEN2010}. The crucial observation is that a unitary transformation preserves the condition number but changes the coherence number $\mu(A)$. After prepossessing with row mixing, the coherence number $\mu(A)$ is reduced and uniform sampling can be applied to get a sampled matrix $A_s$ with only $O(n \log(n))$ rows. Then they decomposed the sampled matrix $A_s = QR$ and used LSQR method to solve the original system $Ax = b$ with preconditioner $R^{-1}$~\cite{BLEN2010}. The complexity of row mixing is $O(mn \log(m)) $ and QR decomposition of sampled matrix is $O(n^3)$. LSQR applied to the linear system $Ax = b$ with preconditioner $R^{-1}$ costs $O(mn \kappa(AR^{-1}) |\log(\epsilon)|)$. To conclude, the total complexity of BLENDENPIK method is $O(mn \log(m) + n^3 + 
mn \kappa(A R^{-1}) |\log(\epsilon)|). $

The common feature of these fast least squares solvers is that they all try to use random sampling or random transformation to get a spectrally equivalent matrices but with considerably small size. Then an efficient preconditioner can be constructed via these small size matrices. However, the preprocesses of these methods transform sparse matrices into dense matrices, which cannot take the advantage of sparsity if the original matrix $A$ is.

To conclude, the traditional methods like QR and SVD decomposition need $O(mn^2)$ which is prohibitive when $m,n$ is large. Iterative methods such as CG and Kaczmarz can reduce the complexity if the matrices are well conditioned but failed for the ill conditioned cases. Preconditioner based on randomized row sampling algorithms have been developed but destroy the sparsity. 
Our goal is to construct a preconditioner which can keep the sparsity and improve the poor conditioning for highly overdetermined matrix.

%--------------------------------------------------------------------
\section{Row Sampling}
%------------------------------------------------
\subsection{Preliminaries and Notation}
Before walking into the details of our algorithms, we introduce some notation and concepts we may confront. Suppose $A$ is a matrix of size $m \times n$ with $m \geq n$. Denote $a_1, a_2, \cdots, a_m$ to be the row vectors of $A$ and $a^1, a^2, \cdots, a^n$ the column vectors of $A$. For any vector $v$, $\|v \| = (\sum_i v_i^2)^{\frac{1}{2}}$ is the $l_2$ norm of $v$ and is called norm of $v$ for short. For any matrix $A$, the spectral norm $\| A \| = \max_{x \neq 0} \| Ax \|/\|x\|$ is the induced matrix norm by vector $l_2$ norm and the Frobenius norm $\|A \|_F = (\sum_{i,j} a_{ij}^2)^{\frac{1}{2}}$. 
\begin{definition}[Condition Number]
The condition number $\kappa(A)$ of matrix $A$ is defined as 
$$
\kappa(A) = \frac{\sigma_{\max} (A)}{\sigma_{\min} (A)}
$$
with $\sigma_{\max} (A), \sigma_{\min} (A)$ are the maximal and minimal singular values of $A$ respectively.
\end{definition}
A matrix is said to be singular if the condition number is infinite. In our setting, the matrix $A$ is of full rank. Thus the smallest singular value of $A$ is nonzero and the condition number $\kappa(A) < \infty.$

 Another concept need to mention is the coherence number introduced in~\cite{BLEN2010}. 

\begin{definition}[Coherence Number~\cite{BLEN2010}] Let $A$ be an $m \times n$ full rank matrix, and let $U$ be an $m \times n$ matrix whose columns form an orthonormal basis of the column space of $A$.  The coherence of $A$ is defined as 
$$
\mu(A) = \max_{1 \leq i \leq m} \| u_i \|_2^2,
$$
where $u_i$ is the $i$th row of matrix $U$. 
\end{definition}

Obviously, the coherence number of a matrix is always between $n/m$ and $1$. Matrices with small coherence numbers are called incoherent~\cite{BLEN2010}, for example, the Gaussian matrix $X$ which every element is an independent number generated by standard normal distribution. One example of semicoherent matrices is the one with large coherent number but only half rows have a large norm in the orthogonal factor, for example, 
\begin{equation}\label{SEM}
Y_{m \times n} = \begin{bmatrix} X_{(m - n/2) \times n/2} & 0 \\ 0 & I_{n/2} \end{bmatrix}
\end{equation}
where $I_{k}$ is a square identity matrix of size $k\times k$. 
Coherent matrices are of large coherent number, for instance, 
\begin{equation}\label{COH}
Z_{m \times n} = \begin{bmatrix} I_{n} \\ 0 \end{bmatrix}.
\end{equation}

\subsection{Row Sampling}
We present a row sampling algorithm. Given a matrix $A_{m \times n}$ and a probability mass function $\{ p_k, k =1,2, \cdots, m \} $, which will be called sampling density, randomly choose $s$ rows of $A$ via the given sampling density; see Algorithm~\ref{alg1}. 

\begin{algorithm}
  \caption{The row sampling algorithm introduced in~\cite{FLSA2011}.} 
  \label{alg1}  
\textbf{Input:} $A \in \mathbb R^{m \times n}, b \in \mathbb R^m, $ a probability mass function $\{ p_k, k =1,2, \cdots, m \} $ and a sample size $s$. \\
\textbf{Output:} Sampled matrix $A_s \in \mathbb R^{s \times n} $ 
\begin{algorithmic} 
\FOR {$t = 1:s $} 
\STATE  Pick $i_t \in \{ 1,2,\cdots,m\}$ with probability ${\rm Pr}\{ i_t = k\} = p_k $ in identical and independent distributed (i.i.d.) trials.  
\ENDFOR
\end{algorithmic}
Let $S \in \mathbb R^{s \times m}$ with $S_{t,i_t} = 1/(s p_{i_t})^{1/2}$, then $A_s = SA$ is a sampling of $A$.
\medskip
\medskip
\end{algorithm}

Among various sampling densities, we chose the one proportional to the squared norm of each row:
\begin{equation}\label{R2P}
p_k = \frac{\| a_k \|_2^2}{\| A \|_F^2}, k = 1,2,\cdots, m. 
\end{equation} 
The naive uniform sampling $p_k = 1/m, k = 1,2 \cdots, m$ fails when the coherence number of the matrix is large. For example, for the coherent matrix $Z$ defined in~\eqref{COH}, we have to sample all $s$ rows from the first $n$ rows, whose probability $s! / m^s$ is very tiny, otherwise we will get a rank deficient matrix.

\subsection{Approximation property of the non-uniform sampling}
If we write the normal matrix as the summation of rank $1$ matrices 
$$
A^{\intercal} A = \sum_{i = 1}^m a_i^{\intercal} a_i,
$$
then the approximation obtained by the random sampling is given by 
$$
A_s^{\intercal} A_s = \frac{1}{s}\sum_{t=1}^{s} \frac{1}{p_{i_t}} a_{i_t}^{\intercal} a_{i_t} . 
$$
It is straightforward to verify that $A_s^{\intercal} A_s$ is an unbiased estimator for $A^{\intercal} A$, i.e. 
\begin{equation}\label{EATA}
\mathbb E [A_s^{\intercal} A_s] = A^{\intercal} A,
\end{equation}
for any choice of sampling density. Choice~\eqref{R2P} will minimize the variance in Frobenius norms; see Lemma 1 of Chapter 2 in~\cite{RNLA}. 

More importantly such row sampling density keeps the spectral norm in a small variance with high probability. To show this, we need the following concentration result. 
 
\begin{theorem}[Matrix Bernstein (Theorem 6.1 in \cite{TBRM2012})] Let $\{ X_k \}$ be a sequence of independent random, self- adjoint matrices with dimension $d$. Assume 
$$
\mathbb E [X_k] = 0 \quad \text{and} \quad \lambda_{\max}(X_k) \leq R \quad \text{almost surely.}
$$
Let
$$
\sigma^2 = \| \sum_k {\rm Var}(X_k) \| = \| \sum_k \mathbb E(X_k^2) \|. 
$$
Then for all $t \geq 0$
\begin{equation}\label{MatBer}
{\rm Pr} \left( \lambda_{max}(\sum_k X_k) \geq t \right) \leq d \exp \left( - \frac{t^2/2}{\sigma^2 + Rt/3} \right). 
\end{equation}
\end{theorem}
From the matrix Bernstein inequality, we can derive the following result which is  slightly different with results in~\cite{Oliveira:2010random}. 
\begin{corollary}[Sum of Rank-1 Matrices]\label{SR1}
Let $y_1,y_2,\cdots,y_s$ be i.i.d. random column vectors in $\mathbb R^n$ with 
$$
\| y_k \| \leq M \quad \text{and} \quad \|\mathbb E[y_k y_k^{\intercal}] \| \leq \alpha^2,
$$
for $k=1,2,\ldots, s$. 
Then for any $\epsilon \in [0,1]$
$$
{\rm Pr} \left( \| \frac{1}{s} \sum_{k=1}^s y_k y_k^{\intercal} - \mathbb E [y_1 y_1^{\intercal}] \| \geq \epsilon \right) \leq 
2n \exp \left(-\frac{3 s \epsilon^2}{(6 \alpha^2 + 2 \epsilon) (M^2+\alpha^2)}\right). 
$$
\end{corollary}
\begin{proof}
 Let $Y_k =  y_ky_k^{\intercal}, A =  \mathbb E\left [ Y_k\right ]$, and $X_k = (Y_k - A)/s$ for $k = 1,2, \ldots, s$. Then $\mbb E [X_k]=0$. We bound the spectral norm of $X_k$ as
$$
\lambda_{\max}(X_k) = \|X_k\| \leq \frac{1}{s} \left ( \|Y_k\| + \|A\| \right ) \leq \frac{ M^2 + \alpha^2}{s},
$$ 
where we use the fact $\| Y_k\| = \|y_ky_k^{\intercal}\| = \|y_k\|^2 \leq M^2$. We then compute the variance 
\begin{align*}
\|\mbb E[X_k^2]\| = \|\var(X_k)\| = \frac{1}{s^2} \|\var (Y_k)\| = \frac{1}{s^2} \| \mbb E[Y_k^2] - A^2\|
%\leq\frac{1}{s^2}(  \|  \mbb E\left [\|y_k\|^2Y_k\right ] \|+ \| A^2 \| ) 
%\leq \frac{1}{s^2} (\| \| y_k \| \mbb E [\left[Y_k] \right] + \| A\|^2) 
  \leq      \frac{\alpha^2( M^2 + \alpha^2)}{s^2}.
\end{align*}
Here we compute $Y_k^2 = y_ky_k^{\intercal}y_ky_k^{\intercal} = \|y_k\|^2Y_k$. Sum over $k$ to get $\sigma^2 \leq \alpha^2(M^2 + \alpha^2)/s.$

Plug the bound $R \leq (M^2 + \alpha^2)/s, \sigma^2 \leq \alpha^2(M^2 + \alpha^2)/s$ into inequality 
(\ref{MatBer}) and rearrange the terms, we get the desired result. 

\end{proof}

We shall apply this concentration result to our row sampling scheme. 
\begin{theorem}\label{th:approx}
Suppose $A$ is a matrix of size $m\times n$ with $m > n$ and $\|A\|_F^2 = n$. Given $\epsilon \in (0,1), \delta \in (0,1)$, let $C = \frac{2}{3} (6 \| A \|^2 + 2 \epsilon)(1 - \log_n (\delta/2))$ and $s = C\epsilon^{-2}n\log n$. Let $A_s$ be a sampled matrix obtained by Algorithm 1 with sampling density~\eqref{R2P} and sample size $s$. Then 
\begin{equation}\label{eq:spectralnorm}
\| A_s^{\intercal} A_s - A^{\intercal} A \| \leq \epsilon  \quad \text{with probability at least $1 - \delta$}.
\end{equation}
\end{theorem}
\begin{proof}
 Let $y$ be a random variable taking value $a_i^{\intercal}/ \sqrt{p_i}$ with probability $p_i,  1 \leq i \leq m.$ And $y_k, k=1,2,\ldots, s$ be i.i.d. copies of $y$.  
 Then 
$$
A_s^{\intercal} A_s = \frac{1}{s} \sum_{k=1}^{s} y_k y_k^{\intercal},
$$
and
$$
\mathbb E [y_k y_k^{\intercal}] = \sum_{i=1}^m a_i^{\intercal} a_i = A^{\intercal} A.
$$
Thus we have the bound
$$
\| \mathbb E [y_k y_k^{\intercal}] \| = \| A^{\intercal} A \| = \lambda_{\max} ( A^{\intercal} A ) = \| A \|^2 \leq \| A \|_F^2 = n,
$$
and the bound
\begin{align*}
\| y_k \| &\leq \max_{1\leq i\leq m} \frac{\| a_i \|}{\sqrt{p_i}} = \| A \|_F = \sqrt{n}. \\
\end{align*}

By Corollary~\ref{SR1}, for all $0 \leq \epsilon \leq 1$
$$
{\rm Pr} \left( \| \frac{1}{s} \sum_{k=1}^s y_k y_k^{\intercal} - \mathbb E [y_1 y_1^{\intercal}] \| \geq \epsilon \right) \leq 
%2n \exp \left(-\frac{3 C n \log(n) \epsilon^2}{16 n} \right). 
2n \exp \left( - \frac{3C n \log(n) }{ (6 \|A\|^2 + 2 \epsilon)(n + \|A\|^2)} \right). 
$$
i.e.
\begin{align*}
{\rm Pr} \left( \| A_s^{\intercal} A_s - A^{\intercal} A \| \geq \epsilon \right)  &\leq  2n \exp \left( - \frac{3 C n \log(n) }{(6 \|A\|^2 + 2 \epsilon) (n + \|A\|^2)} \right) \\
&\leq  2n \exp \left( -\frac{ 3 C \log(n)  }{ 2(6 \|A\|^2 + 2 \epsilon) } \right) \\
&=  \delta.
\end{align*}
%To satisfy ${\rm Pr} \left(  \| A_s^{\intercal} A_s - A^{\intercal} A \| \geq \epsilon \right)  < \delta $, we need 
%$$C \geq \frac{2}{3} (6 \|A\|^2 + 2 \epsilon) (1 - \log_n(\delta/2)). $$
\end{proof}

\begin{remark} \rm 
Constant $C$ used in Theorem~\ref{th:approx} is not practical since the lower bound of $C$ is quit big. For example, when $\delta = \frac{1}{20}, \epsilon = \frac{1}{2}, \| A \| \leq1$ and $n = 300$, $C \epsilon^{-2}$ should be greater or equal than $\frac{56}{3}(1+ \log_{300}(40)) \approx 30.74$. In practice, however, $C \epsilon^{-2} =4$ is good enough to get a reasonable sampling matrix. $\Box$
\end{remark}

\begin{corollary}\label{cor:spectrumbound}
With the same setting in Theorem~\ref{th:approx}, the following bound hold with high probability $1-\delta$
$$
\lambda_{\min}(A^{\intercal}A) - \epsilon \leq \lambda_{\min}(A_s^{\intercal}A_s) \leq \lambda_{\max}(A_s^{\intercal}A_s)\leq \lambda_{\max}(A^{\intercal}A) + \epsilon.
$$ 
\end{corollary}
\begin{proof}
By the triangle inequality, we immediately get 
\begin{equation}\label{eq:bound}
\| A_s^{\intercal}A_s x \| \leq \| A^{\intercal}A x \| + \epsilon \| x \| \leq \left (\lambda_{\max}(A^{\intercal}A) + \epsilon \right )  \|x \|,
\end{equation}
which implies the desired inequality as $A_s^{\intercal}A_s$ is symmetric. The lower bound of $\lambda_{\min}(A_s^{\intercal}A_s)$ can be proved similarly.
\end{proof}
Notice that the spectrum bound obtained in Corollary~\ref{cor:spectrumbound} will not imply the bound of the preconditioned system $(A_s^{\intercal}A_s)^{-1}A^{\intercal}A$, which requires comparison of $(A_s^{\intercal}A_s x,x)$ and $(A^{\intercal}A x, x)$ for all $x\in \mathbb R^n$.

Next we shall establish some inequalities restricted to high frequency. We call $x \in \mathbb{R}^n$ is a high frequency of matrix $A^{\intercal}A$ if the inequality
\begin{equation}\label{highf}
\lambda_{\max} (A^{\intercal}A) \| x \|^2  \leq C_h (A^{\intercal}A x, x),
\end{equation}
holds with a universal constant. Consider the decomposition of $x$ using the eigen-vector bases of $A^{\intercal}A$. Inequality~\eqref{highf} implies $x$ is mainly expanded by eigen-vectors whose eigenvalues are large. The constant $C_h$ in (\ref{highf}) is introduced to include not only the highest frequency (corresponding to the largest eigenvalue) but a range of frequencies comparable to the highest one. 

\begin{corollary}\label{cor:bound4high}
With the same setting in Theorem~\ref{th:approx}, the following bound hold with high probability $1-\delta$: for all high frequency vectors $x$
$$
\left (1 - C_h \epsilon \right )(A^{\intercal}Ax, x) \leq (A_s^{\intercal}A_s x, x)\leq \left (1 + C_h \epsilon \right )(A^{\intercal}Ax, x).
$$
\end{corollary}
\begin{proof}
By the triangle inequality and the definition of high frequency vectors, we will have
$$
(A_s^{\intercal}A_s x, x)\leq (A^{\intercal}Ax, x) + \epsilon(x,x) \leq \left [1 + \frac{C_h \epsilon}{\lambda_{\max}(A^{\intercal}A)}\right ](A^{\intercal}Ax, x),
$$
Now we can bound $\lambda_{\max}(A^{\intercal}A) \geq 1$ by
$$
n \lambda_{\max}(A^{\intercal}A) \geq \sum_{i=1}^n \lambda_{i}(A^{\intercal}A) = \sum_{i=1}^n \sigma_i^2(A) = \| A \|_F^2 = n.
$$
The lower bound can be proved similarly 
%$$
%\left [1 - \frac{C_h \epsilon}{\lambda_{\max}(A^{\intercal}A)}\right ](A^{\intercal}Ax, x) \leq (A_s^{\intercal}A_s x, x).
%$$
\end{proof}
Corollary~\ref{cor:bound4high} implies $(A_s^{\intercal} A_s)^{-1}$ is an effective smoother for $A^{\intercal} A.$ 
Since Gauss-Seidel iteration can smooth out the high frequency very quickly, we apply several symmetric Gauss-Seidel iterations instead of computing $(A_s^{\intercal} A_s)^{-1}$ in practice. 

To illustrate the approximation property of the sampled matrix, we plot the graph of $A^{\intercal} A$ and $A_s^{\intercal} A_s$ below. 
\begin{figure}[htp]
\centering
    \includegraphics[width=5in,height=3in]{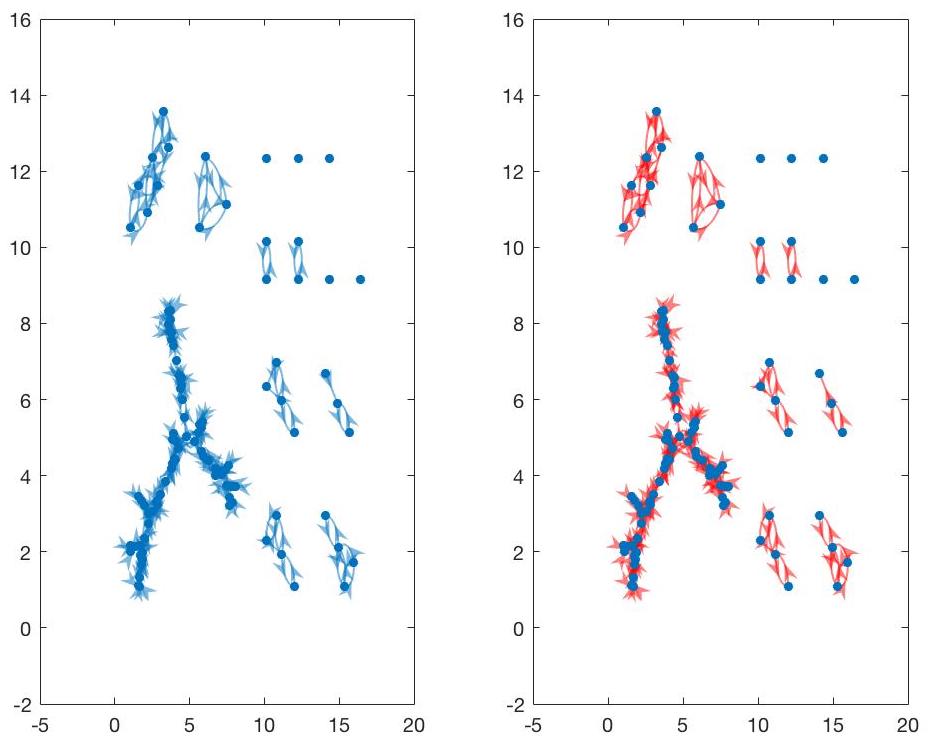}
\caption{Graphs of matrices $A^{\intercal} A$ (left) and $A_s^{\intercal} A_s$ (right). The matrix $A$ is of size $m \times n$, where $m = 9314$ and $n =100$. The sampled matrix $A_s$ is of size $s \times n$ with $s = 1843$ and $n = 100$. The matrix $A$ is rescaled so that the diagonal of $A^{\intercal} A$ is one. The entries which have small absolute values less than a threshold $\theta = 0.125$ in the matrix $A^{\intercal} A$ and $A_s^{\intercal} A_s$ is filtered out and not plot in the graph.}
\end{figure}
The matrix $A$ is of size $m \times n$, where $m = 9314$ and $n =100$. The sampled matrix $A_s$ is of size $s \times n$ with $s = 1843$. The matrix $A$ is rescaled so that the diagonal of $A^{\intercal} A$ is unit. The entries which have small absolute values less than a threshold $\theta = 0.125$ in the matrix $A^{\intercal} A$ and $A_s^{\intercal} A_s$ are filtered out and not shown in the graph. Each edge in the graph represents one entry in the matrix and the thickness of edge represent the magnitude respectively. From the figures, we find out that the two graphs are almost identical, which means the sampling strategy is  able to capture the entries in the normal matrix with a large absolute value which is known as strong connectedness in algebraic multigrid methods~\cite{brandt1982algebraic} and related to the high frequency vectors.

\section{PCG with a Preconditioner based on Row Sampling}
In this section, we present our algorithm by constructing a fast, efficient and easy to implement randomized row sampling preconditioner and apply PCG to solve the normal equation. 
\subsection{Algorithms}
We first normalize the matrix $A$ to make the column vectors have unit length, which enables all diagonal entries of $A^{\intercal} A$ are one and $\|A\|_F^2 = n$. We then apply the row sampling to get a smaller matrix $A_s$ of size $s \times n$ by randomly choosing $s = O(n \log(n))$ rows of the normalized matrix $A_{m \times n}$ with sampling density $p_i = \| a_i\|_2^2/n$. We build our preconditioner by using a few steps of symmetric Gauss-Seidel (SGS) iteration to solve the approximate problem $A_s^{\intercal}A_s e = r$. After all, we apply PCG to the normal equation $A^{\intercal} A x = A^{\intercal} b$ with this preconditioner. 

\begin{algorithm}
\textbf{Input:} $A \in \mathbb R^{m \times n}, b \in \mathbb R^m, $ convergence threshold $\epsilon \in \left( 0, 1\right)$. \\
\textbf{Output:} approximated $\tilde{x}_{\rm{opt}} \in \mathbb R^m. $ 
\begin{enumerate}
\item \textbf{Normalization:} 
$A \leftarrow AD^{-1}$, where $D_{jj} = \|a^j\|_2,$ with $a^j$ being the $j$th column vector of $A$ for $1\leq j \leq n$.  
\item \textbf{Sampling:} Sample the row of $A$ to get $A_s$ with $s = 4 n \log(n)$ by row sampling Algorithm~\ref{alg1} with probability~\eqref{R2P}. 
\item \textbf{Preconditioner:} Construct preconditioner $e= Pr$ by solving $A_s^{\intercal} A_s e =r$ via several symmetric Gauss Seidel iterations; see Algorithm~\ref{alg3}.
\item \textbf{PCG:} Use PCG to solve $A^{\intercal}A x = A^{\intercal}b$ with the preconditioner constructed in Step 3. Stop when the relative residual is below $\epsilon$.
\medskip
\end{enumerate}
  \caption{Randomized Sampling Preconditioned PCG}
  \label{alg2}
\end{algorithm}

For easy of understanding and completeness, the symmetric Gauss-Seidel method is presented below in Algorithm~\ref{alg3}. 

\begin{algorithm}
  \caption{The Preconditioner using Symmetric Gauss-Seidel Iterations.}
  \label{alg3}
\textbf{Input}:  Sampled matrix $A_s \in \mathbb R^{s \times n}$, residual $r \in \mathbb R^n$, number of symmetric Gauss-Seidel iterations $t \in \mathbb Z^+$.\\
\textbf{Output}:  Correction $e \in \mathbb R^n.$ 
\begin{algorithmic}
\FOR {$i = 1:t $} 
\STATE
$$ e \leftarrow e + B^{-1}(r- A_s^{\intercal} A_s e) $$
with $B$ the lower triangular part of $A_s^{\intercal} A_s$. 
\ENDFOR
\FOR {$i = 1:t $} 
\STATE
$$ e \leftarrow e + (B^{\intercal})^{-1}(r- A_s^{\intercal} A_s e)$$
with $B^{\intercal}$ the upper triangular part of $A_s^{\intercal} A_s$. 
\ENDFOR
\end{algorithmic}
\medskip
\end{algorithm}

%------------------------------------------------

\subsection{Complexity}
We compute the complexity of PCG with the randomized sampling preconditioner for both dense matrices and sparse matrices. We use $s = 4n\log(n)$ as the default sample size. Here the factor $4$ is chosen to balance the set up time and solver time. Similarly the number of SGS is set as $5$ to balance the inner iteration of preconditioner and outer iteration of PCG.
%For example, in the `sprand' case with $m = 90000, n = 300.$
%If the sampling size is decreased to $s = 2 n \log(n) \approx 3432$, the PCG iterations increase by about $15\%$. If the sampling size is increased to $s = 8 n \log(n) \approx 13689$, iterations would be $16\%$ less. Thus, we choose $s = 4n\log(n)$ as an optimal point balancing the sampling size and iterations in order to minimize the computation cost.

For dense matrices, in the normalization step, we need $O(mn)$ to calculate the norm of each column $\|a^j\|_2 $ and $O(mn)$ for the matrix multiplication $AD^{-1}$. Sampling costs $O(s n) = O(n^2 \log(n))$. The matrix multiplication $A_s^{\intercal} A_s$ costs $O(s n^2) = O(n^3 \log(n)). $  The preconditioner in PCG, i.e. several symmetric Gauss Seidel iterations for $A_s^{\intercal}A_s e =r $ needs $O(n^2)$. Finally,  PCG iteration steps $k = O(|\log(\epsilon)| \kappa(PA))$ until reaching tolerance $\epsilon$ costs $O(k mn) = O(|\log(\epsilon)| \kappa(PA) mn)$. Note that since we only do matrix vector multiplication with $Ax$ and $A^{\intercal} (Ax)$ instead of matrix product $A^{\intercal} A$, the computation cost for each PCG step is only $O(mn)$ not $O(mn^2)$. Thus the total complexity is $O(|\log(\epsilon)| \kappa(PA) (mn + n^2))+ O(n^3 \log(n))$ when $A$ is dense. 

%Sampling costs $O(s n) = O(n^2 \log(n))$. The matrix multiplication $A_s^{\intercal} A_s$ costs $O(s n^2) = O(n^3 \log(n)). $  The preconditioner in PCG, i.e. several symmetric Gauss Seidel iterations for $A_s^{\intercal}A_s e =r $ needs $O(n^2)$. Finally,  PCG iteration steps $k = O(|\log(\epsilon)| \kappa(PA))$ until reaching tolerance $\epsilon$ costs $O(k mn) = O(|\log(\epsilon)| \kappa(PA) mn)$. Note that since we only do matrix vector multiplication with $Ax$ and $A^{\intercal} (Ax)$ instead of matrix product $A^{\intercal} A$, the computation cost for each PCG step is only $O(mn)$ not $O(mn^2)$. Thus the total complexity is $O(|\log(\epsilon)| \kappa(PA) (mn + n^2))+ O(n^3 \log(n))$ when $A$ is dense. 
 
Complexity would be reduced significantly when the matrix is sparse. Let $nnz(M)$ be the number of nonzero elements of matrix $M$. In the normalization step, the cost is reduced to $O(nnz(A))$ for both the column calculation and matrix multiplication $AD^{-1}$. In the sampling step, depending on the sparse pattern, the complexity is reduced to at most $O(nnz(A))$. 
The matrix product of $A_s^{\intercal} A_s$ costs between $O(nnz(A_s))$ and $O(n \cdot nnz(A_s))$.
The preconditioner costs $O(nnz(A_s^{\intercal} A_s))$. And $k= O(|\log(\epsilon)| \kappa(PA))$ PCG iterations needed. The total complexity is 
$O(|\log(\epsilon)| \kappa(PA) (nnz(A)  + nnz(A_s^{\intercal} A_s))) + O(\alpha nnz(A_s))$ for sparse matrices, where $\alpha \in [1,n]$ depends on the sparse pattern of $A_s$. For sparse matrix $A$ with $nnz(A)\ll mn$, the proposed solver is thus more efficient. 

\begin{table}
\caption{Complexity of Algorithm 2 for Dense and Sparse Matrices}
\begin{center}
\begin{tabular}{|c|c|c|}
\hline
& Dense Matrix & Sparse Matrix\\
\hline 
Normalization & $O(mn)$ & $O(nnz(A))$\\
\hline
Sampling & $O(n^2 \log(n))$ & $O(nnz(A))$ \\
\hline
$A_s^{\intercal}A_s$ & $O(n^3 \log(n))$ & $O(nnz(A_s))$ to $O(n \cdot nnz(A_s))$\\
\hline
Preconditioner & $O(n^2)$ & $O(nnz(A_s^{\intercal} A_s))$ \\
\hline
CG iteration & $O(|\log(\epsilon)| \kappa(PA) mn)$ & $O(|\log(\epsilon)| \kappa(PA) nnz(A))$\\
\hline
\end{tabular}
\end{center}
\label{default}
\end{table}%

Theoretically we cannot find a uniform control of the condition number of the preconditioned matrix $PA$ but we shall show in the next section that numerically our preconditioner is effective. 

%------------------------------------------------

\section{Numerical Results}

We shall compare PCG with our randomized sampling preconditioner, denoted by suffix RS, with CG for the normalized matrix, denoted by suffix CG, which is equivalent to use PCG for the original matrix with a diagonal preconditioner. The column with prefix `Setup' in tables is the CPU time for preprocess including sampling and normalization for RS and only normalization for CG. The column `Time' is the CPU time for iterative methods. Thus the sum of these two are the CPU time for the whole procedure. The column `$\kappa(A^{\intercal} A)$' lists the condition number of $A^{\intercal} A$ and $\mu(A)$ is the coherence number with normalized $A$. We list the coherence number here to emphasize the weighted row sampling works well and robust to the coherence number. It is shown in \cite{BLEN2010} that the uniform sampling fails when the coherence number of the matrix is large. In our sampling algorithm, we use the sampling density proportional to the squared norm of each row. Tolerance for PCG or CG is set $10^{-7}$. Notice that iterative methods may end without reaching the tolerance. 

When varying the size of the matrix, we report one realization of the sampling algorithm. 
As the sampling contains randomness, for each example, we shall also pick up one typical matrix and run our solver $10$ times and compute the mean and standard derivation. 

We tested several classes of matrices -- including well conditioned matrices, ill conditioned matrices, incoherent matrices and coherent matrices; see Table~\ref{table:allmatrx}. 

\begin{table}
\centering
\caption{Classes of Matrices}
\label{table:allmatrx}
\begin{tabular}{|*{3}{c|}}\hline
 & incoherent & coherent \\\hline
 well conditioned &Gaussian (Example 1) & semi Gaussian  (Example 2)\\\hline
 ill conditioned & UDV, sprand  (Example 3, 4) &  random graph Laplacian  (Example 5)\\\hline
\end{tabular}
\end{table}

\subsection{Gaussian Matrix}
The Gaussian matrix is constructed by MATLAB command \texttt{A = randn(m,n)} with each entry of $A$ being generated independent and identically by a standard normal random variable. The matrix $A^{\intercal}A$ has a small condition number followed by Bai and Yin~\cite{EIGSC2008}. 
%
%Theorem 1 in~\cite{EIGSC2008} claims that if a matrix $X_{p \times n}$ with each element being a random number with mean zero and variance $1$ generated independent and identically, and let 
%$$
%S = \frac{1}{n} X X^{\intercal}. 
%$$  
%Then, if $\mathbb E|X_{11}|^4 \leq \infty,$ as $n \rightarrow \infty, p \rightarrow \infty, p/n \rightarrow y \in (0,1),$
%$$
%\lim_{n \rightarrow \infty} \lambda_{\min} (S)= (1 - \sqrt{y})^2 \quad \text{a.s.}
%$$
%$$
%\lim_{n \rightarrow \infty} \lambda_{\max} (S)= (1 + \sqrt{y})^2 \quad \text{a.s.}
%$$
%where $\lambda_{\min}$ and $\lambda_{\max}$ are the smallest and largest eigenvalues of $S$ respectively. 
%
%In our case, $\frac{1}{m} A^{\intercal} A$ is the sample covariance matrix. 
By Theorem 2 in~\cite{EIGSC2008}, the limit of condition number of $A^{\intercal} A$ can be calculated as
$$
\kappa(A^{\intercal} A) = \frac{\lambda_{\max} (A^{\intercal} A)}{ \lambda_{\min} (A^{\intercal} A)} 
\rightarrow \frac{m (1 + \sqrt{n/m})^2 }{ m (1 - \sqrt{n/m})^2 } = \frac{\sqrt{m} + \sqrt{n}}{\sqrt{m} - \sqrt{n}}.
$$
%
%Thus $A^{\intercal} A$ is well conditioned as long as the matrix size $n,m$ is large enough, and $A$ has a rectangular shape, i.e. $m \gg n$. 
When $m = n^2,$ $\kappa(A^{\intercal} A) \approx (\sqrt{n} +1)^2/(\sqrt{n}-1)^2 \approx 1 + O(1/\sqrt{n})$ almost surely.

Since each element is generated independent and identically, the $Q$ factor of $A$'s $QR$ decomposition has evenly distributed magnitude in each row. Thus the coherence number $\mu(A)$ of Gaussian matrix is also small. In summary the Gaussian matrix belongs to the category-- `well conditioned and incoherent matrices'. The sampling density~\eqref{R2P} is almost uniform.

\begin{table}[ht]
\centering
\caption{Gaussian Matrix: Residual and Iteration Steps}
\begin{tabular}{|*{9}{c|}}
\hline
$n$   & $m$    & $nnz(A)$    & $\kappa(A^{\intercal} A)$ & $\mu(A)$   & Residual.CG & Iter.CG & Residual.RS & Iter.RS \\\hline
    109 &    3000 &     11881 &   8.39  &  0.05 &   5.03e-08  &   10     &    7.00e-08  &   11  \\\hline   
    141  &   5000   &   19881   &  8.30  &  0.01 &   8.60e-08  &    9    &     5.19e-08 &    11  \\\hline   
    200  &  10000 &     40000  &   7.70   &  0.02  &  1.70e-08  &    9     &     4.42e-08    & 11  \\\hline	   
    282  &  20000 &     79524   & 7.22   &  0.05  &  4.03e-08    &  8    &     2.70e-08   &  11    \\\hline 
    400  &  40000  &  160000  &   7.40   &  0.09 &   9.62e-08   &   7   &      2.64e-08  &   11  \\\hline   

\end{tabular}
\end{table}
\hfill
\begin{table}[!htbp]
\centering
\caption{Gaussian Matrix: Elapsed CPU Time}
\begin{tabular}{|*{8}{c|}}\hline
$n$   & $m$    & Time.CG & Setup.CG & Sum.CG & Time.RS & Setup.RS & Sum.RS \\\hline
     109 &    3000 &   2.56e-03 &    2.32e-03 &  4.88e-03 &  3.33e-03 &   6.33e-03   &  9.66e-03 \\\hline
    141  &   5000  &  5.34e-03 &     5.56e-03 &  1.09e-02 &  8.32e-03 &   1.30e-02    & 2.13e-02 \\\hline
    200  &  10000  &  1.90e-02 &   1.95e-02 &    3.84e-02 &  2.51e-02 &   4.05e-02 &   6.56e-02 \\\hline
    282  &  20000  &   3.46e-02 &    4.34e-02 &   7.80e-02 &   4.09e-02 &   7.31e-02 &    1.14e-01 \\\hline
    400  &  40000  &   1.00e-01 &    1.26e-01 &    2.27e-01 &   1.43e-01 &   2.06e-01 &     3.48e-01 \\\hline
\end{tabular}
\end{table}

\begin{table}[!htbp]
\centering
\caption{Gaussian Matrix: Mean and Sample Standard Deviation}
\begin{tabular}{|*{8}{c|}}\hline
$n$   & $m$  &  Iter.Mean &  Iter.Std &  Time.Mean &  Time.Std   &  Setup.Mean     &    Setup.Std \\\hline
     109 &    3000 &     11    &  0   &   3.37e-03   & 5.39e-04 &  5.64e-03   &           9.15e-04\\\hline	
    141  &   5000  &    11     &   0    &   1.0e-02    &1.58e-03   &   1.40e-02     &       1.80e-03\\\hline
    200  &  10000  &   11     &    0   &    1.47e-02   &   2.62e-03  &   2.56e-02    &       4.18e-03\\\hline
     282  &  20000  &  11      &     0    &  4.04e-02    &  3.59e-03   & 7.25e-02      &      7.14e-03\\\hline
      400  &  40000  &   10.9     &    0.31 &   1.40e-01   &  5.94e-03   &    2.12e-01     &      1.42e-02\\\hline
\end{tabular}
\end{table} 

%------------------------------------------------
\subsection{`Semi Gaussian' Matrix}
The `semi Gaussian' matrix used in~\cite{BLEN2010} has the following block structure. The left upper block $B$ is a Gaussian matrix of size $(m-n/2) \times n/2$ and the right lower block $I_{n/2}$ is an identity matrix of size $n/2 \times n/2$. 
$$
A_{m \times n} =
 \begin{bmatrix}
  B & 0 \\
   0 & I_{n/2}
 \end{bmatrix}. 
$$
It belongs to the category-- `well conditioned and coherent matrices'. 
 
For such `semi Gausian' matrices, the coherence number $\mu(A) = 1$ due to the presentness of an identity sub-matrix. It is shown in~\cite{BLEN2010} that the uniform sampling fails for this example while our non-uniform sampling works well. To use random transformations, a small perturbation $10^{-8}$ is added to every entry of $A$ to change it to a dense matrix.

They are also well conditioned since 
$$
A^{\intercal} A = \begin{bmatrix} B^{\intercal} & 0 \\ 0 & I \end{bmatrix} 
\begin{bmatrix} B & 0 \\ 0 & I \end{bmatrix} = 
\begin{bmatrix} B^{\intercal} B & 0 \\ 0 & I \end{bmatrix},
$$
and when $\lambda_{\max}(B^{\intercal} B)\geq 1$
$$
\kappa(A^{\intercal} A)  \leq \kappa(B^{\intercal} B). 
$$
The Gaussian matrix $B$ is well conditioned by the analysis in \S 5.1. So is $A$.

\begin{table}[!htbp]
\centering
\caption{`Semi Gaussian' Matrix: Residual and Iteration Steps}
\label{table_semi}
\begin{tabular}{|*{9}{c|}}
\hline
$n$   & $m$    & $nnz(A)$    & $\kappa(A^{\intercal} A)$ & $\mu(A)$   & Residual.CG & Iter.CG & Residual.RS & Iter.RS   \\\hline
    62   &  1000   &   992   & 1.99e+03 & 1  &  2.25e-08  &   9   &       6.85e-08  &   11      \\\hline
    108  &   3000  &   2970  &  5.82e+03  &  1 &  6.87e-08  &   8  &    3.70e-08  & 12        \\\hline
    140   &  5000   &  4970   &   9.80e+03   & 1   &     2.43e-08  &   8    &      6.09e-08   &  11     \\\hline
    200   & 10000   & 10100  &   1.95e+04 &  1 &    6.48e-08  &   7     &     6.79e-08    & 11       \\\hline
    282   & 20000   & 20022    & 3.84e+04  & 1  &     2.28e-08 &   7    &   7.37e-08   &  11       \\\hline
\end{tabular}
\end{table}
\begin{table}[htb]
\centering
\caption{`Semi Gaussian' Matrix: Elapsed CPU Time}
\begin{tabular}{|*{8}{c|}}\hline
$n$   & $m$        & Time.CG & Setup.CG & Sum.CG & Time.RS & Setup.RS & Sum.RS  \\\hline
     62  &   1000  &      7.03e-04   &    3.96e-04    & 1.10e-03 &     1.78e-03   &    2.16e-03  &3.94e-03    \\\hline
    108  &   3000  &     2.31e-03  &       2.95e-03  & 5.26e-03&      4.03e-03 &      7.25e-03  & 1.13e-02   \\\hline
    140  &   5000  &      5.14e-03    &   6.31e-03  &   1.14e-02&    9.95e-03 &     1.49e-02  & 2.49e-02   \\\hline
    200  &  10000 &      1.54e-02   &    1.77e-02  &    3.31e-02&    2.47e-02 &      3.84e-02  & 6.31e-02  \\\hline
    282  &  20000 &      3.19e-02 &     4.28e-02   &    7.47e-02&  4.36e-02 &      7.83e-02  &  1.22e-01  \\\hline
\end{tabular}
\end{table}

\begin{table}[!htbp]
\centering
\caption{`Semi Gaussian' Matrix: Mean and Sample Standard Deviation}
\begin{tabular}{|*{8}{c|}}\hline
$n$   & $m$        &  Iter.Mean  &Iter.Std  &   Time.Mean &  Time.Std   &  Setup.Mean   &   Setup.Std \\\hline
 62  &   1000  &    11.7  & 0.68 &       1.18e-03 &   7.75e-05  &  1.35e-03  &      4.27e-04 \\\hline
 108  &   3000  &     11.7    & 0.48 &      2.65e-03 &   1.30e-04   &  5.31e-03   &      3.12e-03 \\\hline
 140  &   5000  &11.9  &  0.57 &      6.36e-03 &   8.77e-04 &   1.14e-02   &        4.87e-03 \\\hline
200  &  10000 &  11.7   & 0.67 &         1.51e-02 &    2.04e-03  &    2.69e-02  &          5.00e-03 \\\hline
282  &  20000 &  11.3    &    0.48 &     4.51e-02 &   6.93e-03 &   7.69e-02   &          8.21e-03 \\\hline
\end{tabular}
\end{table} 

\newpage

%------------------------------------------------
\subsection{`Sprand' Matrix}
The `sprand' (sparse random) matrix is generated by Matlab function \texttt{A = sprand(m,n,s,1/c),}
where $m$ is the number of rows, $n$ is the number of columns, $s$ is the sparsity and $c$ is the estimated condition number.  We can control the condition number by the input $c$. When $c$ is large, the generated matrix is ill conditioned. The coherence number is still small due to the randomness. Thus it belongs to the category  `ill conditioned and incoherent matrices'.

To test the robustness to the condition number, we fix $m = 90000, n =300$ and the sparsity $s = 0.25$ and change $c$ to get several matrices with large condition number; see Table~\ref{table:sprand}-\ref{table:sprandTime}. 
\begin{table}[!htbp]
\centering
\caption{`Sprand' Matrix $m = 90000, n =300$: Residual and Iteration Steps.}
\label{table:sprand}
\begin{tabular}{|*{7}{c|}}
\hline
 $nnz(A)$    & $\kappa(A^{\intercal} A)$ & $\mu(A)$   & Residual.CG & Iter.CG & Residual.RS & Iter.RS \\\hline
 87248 &       3.87e+03 &   7.17e-03 &   9.05e-08 & 98 & 9.97e-08 & 21    \\\hline 
  87208  &       1.91e+04  &  5.86e-03  & 8.70e-08 & 181 & 6.80e-08 & 38  \\\hline
 86278   &      7.55e+04  & 3.91e-03   & 8.36e-08 & 264 & 7.94e-08 & 58 \\\hline
 86654   & 2.89e+05 &  7.81e-03  &  9.17e-08 & 233 & 9.37e-08 & 50 \\\hline
  86816   & 7.40e+05  &  7.71e-03  &  8.18e-07 & 296 & 4.77e-08 & 70 \\\hline
\end{tabular}
\end{table}

\begin{table}[!htbp]
\centering
\caption{`Sprand' Matrix $m = 90000, n = 300$: Elapsed CPU Time}
\label{table:sprandTime}
\begin{tabular}{|*{7}{c|}}\hline
 $nnz(A)$  & Time.CG & Setup.CG & Sum.CG & Time.RS & Setup.RS& Sum.RS \\\hline
87248 &    1.66   &  0.35  &    2.01  &  0.52  &  0.68    &  1.20 \\\hline
87208 &    3.23   &  0.43  &   3.65  &  0.88  &   0.65   &  1.53\\\hline
86278 &  4.39  &    0.33  &    4.72  &   1.29  &  0.57  &   1.86 \\\hline
86654 &  3.92 &    0.35  &   4.27  &   1.13   & 0.56  &   1.69 \\\hline
86816 &  5.00 &    0.33 &     5.33  &   1.56   & 0.54   & 2.10 \\\hline
\end{tabular}
\end{table}

\begin{table}[!htbp]
\centering
\caption{`Sprand' Matrix $m = 90000, n = 300$: Mean and Sample Standard Deviation }
\begin{tabular}{|*{6}{c|}}\hline
  Iter.Mean & Iter.Std  &  Time.Mean &   Time.Std  &   Setup.Mean  &  Setup.Std \\\hline
     23.3   &  1.95 &      0.57 &   4.45e-02 &   0.59  &         3.73e-02 \\\hline
     39    &   1.70 &      0.92 &   4.13e-02 &  0.57   &   1.55e-02 \\\hline
    60.9   &   3.63 &      1.41 &   7.98e-02 &    0.56   &      1.03e-02 \\\hline 
    51.2   &    2.44 &    1.18 &   5.73e-02 &   0.55       &        2.61e-02 \\\hline 
  69.4   &   2.63 &       1.60 &   6.14e-02 &  0.56     &         1.85e-02 \\\hline 
\end{tabular}
\end{table} 

Notice that for very ill-conditioned matrices, CG without preconditioners will not reach the tolerance $10^{-7}$; see row $5$ in Table~\ref{table:sprand}. Although theoretically CG will result in the exact solution within at most $n$-steps, the large condition number causes the instability. Our preconditioner is effective and PCG converges within $100$ steps.

\begin{table}[!htbp]
\centering
\caption{`Sprand' Matrix $m = 40000$: Mean and Sample Standard Deviation }
\begin{tabular}{|*{8}{c|}}\hline
$n$ &$\kappa(A^{\intercal} A)$  &  Iter.Mean &  Iter.Std & Time.Mean &  Time.Std &  Setup.Mean    &  Setup.Std \\\hline
  50    &       28323     &    17.2   &    1.48 &     0.03 &   7.74e-03  &     0.04  &      1.37e-02 \\\hline
    100   &        31278  &      26.3   &   1.70 &     0.09 &  6.29e-03 &      0.06 &          2.42e-03 \\\hline
    200   &        60858   &    54.6   &    2.46 &    0.36 &   2.00e-02 &   0.16   &          5.84e-03 \\\hline
    400   &        88807     &    72.1     &   4.84 &     1.07 &     8.05e-02 &    0.62  &            1.91e-02 \\\hline
    800   &   1.13e+05   &    86.3      &   2.41 &       3.00 &       8.59e-02 &     3.04 &      8.30e-02 \\\hline
\end{tabular}
\label{tb:m40000}
\end{table} 

We fix $m = 40,000,c = 100, s = 0.25$ and vary $n$ in Table~\ref{tb:m40000}. Again our preconditioned PCG works well.

\begin{table}[!htbp]
\centering
\caption{`Sprand' Matrix $n = 200$: Mean and Sample Standard Deviation }
\begin{tabular}{|*{8}{c|}}\hline
$m$ &$\kappa(A^{\intercal} A)$   &  Iter.Mean & Iter.Std &   Time.Mean &  Time.Std &   Setup.Mean      &  Setup.Std \\\hline
 10000 &   58417    &   54.2      &3.82 &    0.11 &    1.12e-02 &   0.07      &           4.95e-03 \\\hline
    20000 &   48309   &     39.5   &   1.78 &  0.14 &   7.75e-03 &    0.10   &       6.81e-03 \\\hline
    40000  &  69855     &     45.2   &  3.36 &   0.32 &    2.55e-02 &  0.17     &     5.40e-03 \\\hline
    70000 &   49447  &   46.8    &2.66 &      0.56 & 2.85e-02 &     0.26    &     9.24e-03 \\\hline
    90000 &   73177   &      56     & 2.21 &   0.86 &   3.08e-02 &  0.34    &     9.22e-03 \\\hline
\end{tabular}
\label{tb:n200}
\end{table} 

We then fix $n = 200$ and $c = 100$ and vary $m$ in Table~\ref{tb:n200}. The iteration steps are almost uniform to $m$. Notice that for fixed $n = 200$, the sample size $s = 4n\log n \approx 4239$ which is a small portion for large $m$.

%------------------------------------------------
\subsection{UDV Matrix}
The UDV matrices are random matrices generated by 
$
A = U D V,
$
where $U$ is an $m \times n$ random orthonormal matrix, $V$ is an $n \times n$ random orthonormal matrix and 
$D = {\rm diag}[1,1+(c-1)/n,\cdots,c]$ and $c$ is the estimated condition number. For this kind of matrices, we can control the condition number by parameter $c$. When $c$ is large enough, it belongs to the category `ill conditioned and incoherent matrices'.
 \\

\begin{table}[!htbp]
\centering
\caption{`UDV' Matrix $m = 90000, n = 300, nnz(A) = 90000$: Residual and Iteration Steps}
\label{table:UDV}
\begin{tabular}{|*{6}{c|}}
\hline
 $\kappa(A^{\intercal} A)$ & $\mu(A)$   & Residual.CG & Iter.CG & Residual.RS & Iter.RS  \\\hline
 5936    & 4.81e-03  &   9.68e-08  &   116   &     7.21e-08 &    24   \\\hline  
  18853   & 4.61e-03  & 9.68e-08   &  202     &   8.81e-08  &   38     \\\hline
   1.44e+05   &4.66e-03 &    4.42e-07 &     294 &       8.44e-08&     68     \\\hline
4.75e+05   & 4.65e-03 &  1.32e-05  &   369    &      7.92e-08   &  86     \\\hline
 1.07e+06   & 4.72e-03  &   9.89e-06   &  253     &   4.44e-08  &   91     \\\hline  
\end{tabular}
\end{table}

Again CG fails to converge for the last three matrices in `UDV' group when the condition number is large while PCG with our proposed preconditioner works well. 

\begin{table}[!htbp]
\centering
\caption{`UDV' Matrix $m = 90000, n = 300, nnz(A) = 90000$: Elapsed CPU Time}
\begin{tabular}{|*{8}{c|}}\hline
 $\kappa(A^{\intercal} A)$ & $\mu(A)$    & Time.CG & Setup.CG & Sum.CG & Time.RS & Setup.RS & Sum.RS \\\hline
 5936    & 4.81e-03  &  2.44 &       0.42 & 2.86   &  0.49 &  0.47 & 0.96 \\\hline 
  18853   & 4.61e-03  &  3.77 &     0.43   & 4.20 & 0.71 &  0.45 & 1.16 \\\hline 
  1.44e+05   &4.66e-03 &   5.47 &   0.42  & 5.89&  1.30 &   0.45 & 1.75\\\hline 
4.75e+05   & 4.65e-03 &    5.88 &    0.44 & 6.32 & 1.57 &   0.45 & 2.02 \\\hline
1.07e+06   & 4.72e-03  &  5.53 &   0.43   & 5.96 &   1.65 &   0.44 & 2.09  \\\hline 
\end{tabular}
\end{table}

\begin{table}[!htbp]
\centering
\caption{`UDV' Matrix $m = 90000, n = 300, nnz(A) = 90000$: Mean and Sample Standard Deviation }
\begin{tabular}{|*{7}{c|}}\hline
  $\kappa(A^{\intercal} A)$ & Iter.Mean & Iter.Std  &   Time.Mean &  Time.Std   &  Setup.Mean   &   Setup.Std \\\hline
        5936    & 23.1     &    0.57 &     0.51   &    7.8e-02  &  0.45 &        3.79e-02 \\\hline
          18853   &  38.8  &       0.63 &    0.76  &  7.51e-02 &  0.43 &      6.91e-03 \\\hline
    1.44e+05  &    72    &      0.82 &      1.42 &      1.67e-01 &  0.48 &       9.76e-02 \\\hline
    4.75e+05   & 86.4  &        0.52 &     1.86 &      2.37e-01 &  0.47 &       5.84e-02 \\\hline
    1.07e+06  &  90.2  &       0.42 &       1.81 &    1.98e-01 &   0.48 &       5.91e-02 \\\hline
\end{tabular}
\end{table} 

\subsection{Random Graph Laplacian Matrix}
In this section, we shall show our least squares solver can be applied to solve the random graph Laplacian problem which has important application in spectral clustering, text mining and web applications etc~\cite{Chung:1996Spectral}.

%\subsubsection{Introduction of Graph Laplacian Matrices}
%\begin{definition}
A simple graph is an undirected graph without multiple edges or loops. It can be described by a set $V = \{v_1, \cdots v_n\}$ of vertices and a set of edges $E = \{ (v_i, v_j) \}$. 
For a vertex $v_i$, $\deg(v_i)$ counts the number of edges attached to it.
The degree matrix $D$ of a graph is a diagonal matrix consisting of degree of each vertex, i.e. $D_{ii} = \deg(v_i)$ for $i=1,2,\ldots, n$. 
The adjacency matrix $A$ of a graph is defined as
$$
A(i,j) = 
\begin{cases}
1 & \text{if }~\left( v_i, v_j \right) \in E ; \\
0 & \text{otherwise}.
\end{cases}
$$

Given a simple graph $G$ with $n$ vertices, its graph Laplacian matrix $L_{n \times n}$ is defined as
$$
L := D - A.
$$
%where $D$ is the degree matrix and $A$ is the adjacency matrix. 
%The elements of $L$ are given by 
%$$
%L(i,j) = 
%\begin{cases}
%\deg(v_i) \,\,\, \text{if} ~ i=j;\\
%-1 \,\,\,\,\, \text{if there is some edge}  ~\left( v_i, v_j \right) \in E ; \\
%0 \,\,\,\,\, \text{otherwise}.
%\end{cases}
%$$
%
Vertices with zero degree are isolated, i.e. not connected to other vertices. For a simple graph without isolated vertex, the normalized graph Laplacian matrix is defined by normalizing the diagonal of $L$
$$
L_N  := I - D^{-1/2} A D^{-1/2}.
$$
Therefore the elements of $L^N$ are given by 
$$
L_N(i,j) = 
\begin{cases}
1 & \text{ if } i = j;\\
-\frac{1}{\sqrt{\deg(v_i) \deg(v_j)}} & \text{ if } i \neq j \text{ and } \left( v_i, v_j \right) \in E; \\
0 & \text{otherwise}.
\end{cases}
$$
We shall focus on solving the normalized graph Laplacian matrix. 

%In order to assemble the normalized graph Laplacian matrix $L_N$, 
We now reformulate the graph Laplacian matrix as the normal matrix of a least square problem. Given a simple graph $G$, endow an orientation to each edge in $E$ to get a directed graph. 
The incidence matrix $B$ is the matrix representation of a directed graph. Let $m$ be the number of edges and recall that $n$ is the number of vertices. Matrix $B$ is of size $m\times n$ and each row represents a directed edge.
Given any edge $e = (u,v) $, we let $s(e) = u$ be the source of $e$ and $t(e)=v$ be the target of $e$, which means the edge $e$ is leaving vertex $u$ and entering vertex $v$, respectively. Let
$$
B(i,j) = 
\begin{cases}
1 \,\,\,\text{if} \,\,\, s(e_i) = v_j;\\
-1 \,\,\,\text{if}\,\,\, t(e_i) = v_j;\\
0 \,\,\,\text{otherwise}. 
\end{cases}
$$
It is easy to verify the relation between degree matrix, adjacency matrix and incidence matrix is 
$$
B^{\intercal} B = D - A. 
$$
And the matrix product $B^{\intercal} B$ is independent of orientation of each edge. Notice that $B$ is sparse with $nnz(B) = 2m$. 
%Since change an orientation is to right multiply $B$ with a diagonal matrix $W$ with elements equal either $1$ or $-1$. Then
%$$
%(WB)^{\intercal} (WB) = B^{\intercal} W^{\intercal} W B = B^{\intercal} B.
%$$
%The last equality comes from $W^{\intercal} W = I. $This also exerts an heuristic idea on the definition of Graph Laplacian matrix. 

%Remark: In the incidence matrix $B$, each row has at most two nonzero elements since for each edge in graph, it has exactly two vertices. 

Obviously $B^{\intercal}B$ is symmetric and semi-positive definite. 
As the sum of each row of $B$ is $0$, the constant vector is in the null space of $B$, i.e. $B \boldsymbol c = 0,$ which implies $L_N$ is singular. Furthermore the multiplicity of zero eigenvalue is the number of connected component $G$.

\subsubsection{Random Graph Model}
%Random graph means this graph was sampled out of a set of graphs according to a probability distribution \cite{random_graph}. 
We shall use the random graph model presented in ~\cite{Chung_2003}. The spectra of adjacency matrix $A$ follows the power law which enables us to construct an ill-conditioned random graph Laplacian matrix. The construction of a random graph follows several steps. Let $(w_1, w_2, \ldots, w_n)$ be a sequence of positive numbers which denotes the expected degree of $i$-th node. Introduce two parameters $d,m$ for the average degree and maximum degree respectively. 
%We follow the paper\cite{Chung_2003} to construct ill-conditioned random graph Laplacian matrix. 

%The two random graph models introduced here are random binomial graph $G(n,p)$ and uniform random graph $G(m,n)$ \cite{random_graph}. Our model is similar to random binomial graph $G(n,p)$ but different from both of them. 

Following~\cite{Chung_2003}, we shall chose a sequence $w = (w_1,\cdots,w_n)$ following a power law: 
$$
w_i = ci^{-1/(\beta -1)} \quad \text{for} ~ i_0 \leq i \leq n+i_0,
$$
where $c$ is a number determined by the average degree $d$ and $i_0$ depends on the maximum degree $m$, i.e.
$$
c = \frac{\beta -2 }{\beta -1} d n^{-1/(\beta -1)}, ~~ i_0 = n \left[ \frac{d (\beta -2)}{m (\beta -1)} \right]^{\beta -1}.
$$

%The parameter 
%$\beta$ is some number determined by the nature of graph. 
%By the consequence of spectra of random graphs with given expected degrees, 
Many massive graphs are power-law graphs with $2 \leq \beta \leq 3$. 
For some specific graphs like Internet graph~\cite{Furedi:1981aa}, $2.1< \beta < 2.4$. Hollywood graph~\cite{Faloutsos_1999} has exponent $\beta \approx 2.3.$ The corresponding normalized graph Laplacian matrix, is well conditioned and relatively easy to solve by CG. 

By our numerical experiment, the larger $\beta$ is, the larger probability we can get an ill conditioned graph Laplacian matrix. In our test, we fix $i_0 = 11$, choose $\beta = 5, d = 30$ for the first incidence matrix $B_1$, and $\beta = 8, d = 5n$ for the second incidence matrix $B_2$, where $n$ is the number for vertices. $B_1$ is the incidence matrix related to a random graph with less edges while it is ill-conditioned. $B_2$ corresponds to a dense random graph which is usually well-conditioned. We glue two graphs together to get an ill-conditioned graph Laplacian  with $O(n^2)$ edges.

\begin{enumerate}
\item 
Construct a random graph $G$ with edge connected vertices $v_i$ and $v_j$ by probability $P(i,j)$, where
$$
P(i,j) = w_i w_j \rho, \quad \text{with } \rho = 1/ \sum_{i=1}^n w_i.
$$
Then the upper triangular part ($i \leq j$) of adjacency matrix is a random Bernoulli matrix
$$
A_u(i,j) = 
\begin{cases}
1 \,\,\,\text{with probability} ~ P(i,j);\\
0 \,\,\,\text{otherwise}.
\end{cases}
$$
Since the adjacency matrix for an undirected graph is symmetric, we set 
$$
A = A_u^{\intercal} + A_u. 
$$

\item 
Extract positions of all nonzero entries from matrix $A^{\intercal} A$ by \texttt{[i,j,s] = find(A'*A)} and the $3$ vectors $i,j,s$ are the coordinate formate to represent a sparse matrix, i.e. $i(k) = i$, $j(k) = j$, and $s(k) = (A^{\intercal} A)_{ij}$. The incidence matrix $B_1$ can be constructed by edge index $(i,j)$ and distribute $1$ for first vertex $v_i$ and $-1$ for second vertex $v_j$. 

\item Construct another incidence matrix $B_2$ using the same procedure but with different parameters $\beta = 8, d = 5n$. The underlying graph of $B_2$ will have more edges while the corresponding Laplacian matrix is well conditioned.

\item Combine the two incidence matrices $B_1$ and $B_2$ with $5$-overlap in vertices, i.e., construct a combination matrix $B_{\rm{com}}$ by 
\texttt{
B{com}(1:m1, 1:n) = B1; ~ B{com}(m1+1:m2, n-4:2n-5) = B2;
}\\
where recall that $B_1$ of size $m_1 \times n$, $B_2$ of size $m_2 \times n$. 
%We want to combine these two graphs with some common vertices. This is the reason we have $5$-overlap here. 
%\\

\item Filter out the isolated vertices in graph corresponding to matrix $B_{\rm{com}}$.
% as a random graph, it is possible that there are vertices with zero degree, i.e., they are isolated and not connected to other vertices. Such isolated vertex will be removed to get a smaller matrix $B_{\rm{new}}$. 
The final incidence matrix $B = B_{\rm{new}}.$
\end{enumerate}

Our algorithm will sample $B$ to get a graph with much fewer edges. A random graph and its sampling are shown in Fig.~\ref{fig:random_graph}.
\begin{figure}[htp]
\centering
    \includegraphics[width=5in,height=3in]{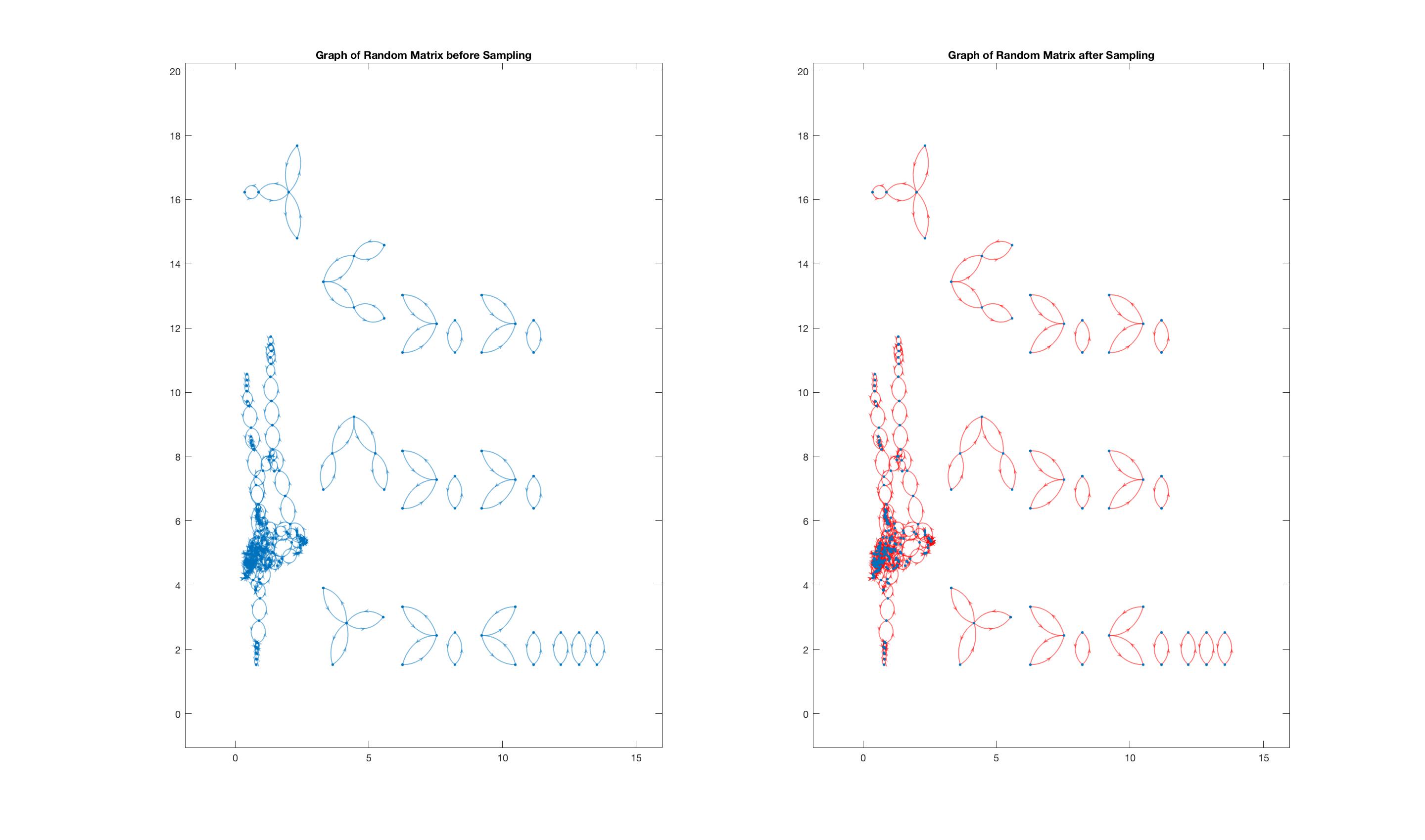}
\caption{Graphs of matrices $B^{\intercal} B$ (left) and $B_s^{\intercal} B_s$ (right). The matrix $B$ is of size $m \times n$, where $m = 68969$ and $n =709$. The sampled matrix $B_s$ is of size $s \times n$ with $s = 18616$ and $n = 709$. The matrix $B$ is rescaled so that the diagonal of $B^{\intercal} B$ is one.}
\label{fig:random_graph}
\end{figure}

\subsubsection{Numerical Results of Graph Laplacian Matrices}
\begin{table}[!htbp]
\centering
\caption{Random Graph Laplacian Matrix: Residual and Iteration Steps}
\begin{tabular}{|*{9}{c|}}\hline
$n$ & $m$ & $s$ & $nnz(B^{\intercal} B)$ & $\kappa(B^{\intercal} B)$    & Residual.CG & Iter.CG & Residual.RS & Iter.RS  \\\hline
    187    &      5078  &   3913    &     10343   & 129.22    & 7.44e-08    &  40     &   5.69e-08  &   17   \\\hline  
    355    &     19475  &   8339   &      39305  &  278.54   & 7.97e-08   &   67   &     9.61e-08 &    25   \\\hline  
    536  &       41032  &  13474  &       82600 &   467.04 &   8.09e-08   &   60   &     8.55e-08  &   18  \\\hline   
    709   &      68969  &  18616 &   138647 &   951.21  &   9.37e-08 &     93  &      9.76e-08 &    25  \\\hline   
    856  &  101829 &   23120 &   204514  &  3145.4&      9.85e-08 &    112   &     6.03e-08 &    28 \\\hline    
\end{tabular}
\label{tb:graphLap}
\end{table}

In Table \ref{tb:graphLap}, we list the iteration steps of CG and PCG with RS preconditioner. As $B^{\intercal} B$ is singular, the condition number in the third column is the so-called effective condition number which is the quotient of the largest eigenvalue over the smallest nonzero eigenvalue of $B^{\intercal} B$. Our least square solver is more robust compared to CG. The iterative steps of our method is less than $1/3$ compared to CG and the elapsed CPU time is comparable to CG, see Table \ref{tb:graphLaptime}. 

\begin{table}[!htbp]
\centering
\caption{Random Graph Laplacian Matrix: Elapsed CPU Time}
\begin{tabular}{|*{9}{c|}}\hline
$n$ & $m$ & $s$ &  Time.CG & Setup.CG & Sum.CG & Time.RS & Setup.RS& Sum.RS \\\hline
   187     &     5078  & 3913  & 1.25e-02 &    3.97e-03 &    1.64e-02 &  1.82e-02 &  1.62e-02 &    3.44e-02 \\\hline
    355    &     19475   & 8339 & 1.10e-02 &   5.73e-03 &    1.67e-02 &  1.17e-02 &  1.74e-02 &     2.91e-02\\\hline
    536   &      41032  &  13474 & 1.55e-02 &  2.78e-03  &  1.82e-02 & 1.09e-02 &  1.31e-02 &       2.40e-02\\\hline
    709  &       68969 & 18616  &4.03e-02 &  3.17e-03 &     4.35e-02 &  2.31e-02 &  1.98e-02 &        4.29e-02 \\\hline
    856  &  101829 & 23120  &  8.68e-02   & 7.98e-03 &     9.48e-02  & 3.90e-02 &   3.69e-02    &     7.57e-02 \\\hline
\end{tabular}
\label{tb:graphLaptime}
\end{table}

\begin{table}[!htbp]
\centering
\caption{Random Graph Laplacian Matrix: Mean and Sample Standard Deviation }
\begin{tabular}{|*{9}{c|}}\hline
$m$ & $n$ & $\kappa(B^{\intercal} B)$ & Iter.Mean &  Iter.Std  & Time.Mean &  Time.Std &   Setup.Mean      &  Setup.Std \\\hline
          5078  &  187  &  1.70e+03    &      16.7    &      1.06  &   3.50e-03 & 7.88e-04   & 2.85e-03  &  3.27e-04 \\\hline
         19475 &   355  &  1.21e+04   &   21.9   &       1.79  &   9.00e-03   &  3.15e-03   & 7.01e-03  &    8.10e-04 \\\hline
         41032  &  536 &   1.85e+04    &   21    &      1.33   &   1.15e-02  &   6.91e-04   & 1.01e-02   &  1.01e-03\\\hline
         68969 &    709   & 5.70e+04   &     26.8    &      2.86   &   2.47e-02   &  4.11e-03   &  2.00e-02   &  1.87e-03 \\\hline
    101829  &  856  &  2.75e+05  &     30.2   &     1.69  &    3.72e-02  &   2.94e-03   &  2.69e-02  &   1.72e-03 \\\hline
\end{tabular}
\end{table} 

%------------------------------------------------
\subsection{Summary of Numerical Results}

It is well known that CG converges fast for well conditioned matrices but fails for ill-conditioned matrices. Although theoretically CG should converge in at most $n$ steps, for ill-conditioned matrices, round off errors are amplified in the orthogonalization procedure which make CG fail to converge within $n$ steps; see Tables~\ref{table:sprand} and~\ref{table:UDV}. With our random sampling preconditioner, PCG converges for all of them. For ill-conditioned matrices, the iterative steps are reduced significantly and elapsed CPU time are reduced to $1/3$ in both the `sprand' group and UDV group. We also apply our least squares solver to solve the random graph Laplacian equation and show it reduces the iteration steps. Although we cannot prove the uniform convergence, the performances listed ahead indicates that random sampling preconditioner is efficient and effective.   

Although the algorithm contains randomness, the standard deviation of iterations and CPU time are acceptable. In most of examples, the ratio of standard deviation to the mean of iterative steps is within $2\%$. Only for the `sprand' cases, the ratio ranges from $2\%$ to $5\%$. 

\section{Conclusion}
In this paper, we construct a randomized row sampling method which aims to solve the least squares problems when matrix $A$ is ill conditioned, sparse, highly overdetermined matrix $m \gg n$. By row sampling with probability proportional to the squared norm of rows, we can get a sampled matrix $A_s$ with size $O(n \log n)\times n$ which can capture the high frequency of the norm matrix. Then the preconditioner is constructed by applying the symmetric Gauss Seidel iteration to the sampled normal matrix $A_s^{\intercal}A_s$. The last but not least is that our preconditioner is very easy to implement compare with the preconditioners based on random transformations. 

\bibliographystyle{abbrv}
%\bibliography{ran_lss}

\end{document}